\documentclass[11pt,eqright]{amsart}
\usepackage{amssymb,amsmath}
\usepackage{bm}
\usepackage{amsfonts}
\usepackage{epsf}
\usepackage{graphicx}
\newtheorem{theorem}{Theorem}[section]
\newtheorem{lemma}[theorem]{Lemma}
\newtheorem{corollary}[theorem]{Corollary}

\theoremstyle{definition}
\usepackage[titletoc]{appendix}
\usepackage{subfigure}
\numberwithin{equation}{section}

\newcommand{\bc}{\begin{center}}
\newcommand{\ec}{\end{center}}
\newcommand{\be}{\begin{eqnarray}}
\newcommand{\ee}{\end{eqnarray}}
\newcommand{\nn}{\nonumber}
\newcommand{\ben}{\begin{eqnarray*}}
\newcommand{\een}{\end{eqnarray*}}

\newcommand{\Om}{\Omega}

\newcommand{\na}{\nabla}

\def\na{\nabla}

\def\cE{\mathcal{E}}

\def\cT{\mathcal{T}}
\def\R{\mathbb{R}}

\def\div{\operatorname{div}}

\newcommand{\reffig}[1]{Figure \ref{#1}}

\makeatletter

\newcommand{\Rmnum}[1]{\expandafter\@slowromancap\romannumeral #1@}
\makeatother

\title[]
{\small Superconvergence of both the Crouzeix--Raviart and Morley elements}
\author[J.~Hu]
{Jun Hu$^\ast$}
\address{$^\ast$ LMAM and School of Mathematical Sciences,
 Peking University, Beijing 100871, P. R. China}
\email{hujun@math.pku.edu.cn}
\author[R. Ma]{Rui Ma$^\dagger$}
\address{$^\dagger$ LMAM and School of Mathematical Sciences,
 Peking University, Beijing 100871, P. R. China}
\email{maruipku@gmail.com}
\thanks{The  first author was supported by  the NSFC Project 11271035 and  by  the NSFC Key Project 11031006.}

\begin{document}
\newpage

\begin{abstract}
In this paper, a new method is proposed to prove the superconvergence of both the Crouzeix--Raviart and Morley elements. The main idea is to fully employ equivalences with the first order Raviart--Thomas element and the first order Hellan--Herrmann--Johnson element, respectively. In this way, some special conformity of discrete stresses is explored and superconvergence of mixed elements can be used to analyze superconvergence of nonconforming elements. Finally, a half order superconvergence by postprocessing is proved for both  nonconforming elements.
\end{abstract}
\maketitle
\section{Introduction}
The superconvergence analysis is well studied for conforming finite elements, see \cite{ChenHuang1995,LinYan1996}, and as well as mixed finite elements of second order problems. For triangular mixed elements, Douglas et al. \cite{DouglasRobert1985} proved superconvergence for the displacement variable on general
triangulations, see also \cite{Arnold}. Brandts \cite{Brandts1994,Brandts2000} proved superconvergence for the stress variable on uniform triangulations for the first and second order Raviart--Thomas elements \cite{RaviartThomas1977}, respectively. For superconvergence along the Gauss--lines in rectangular mixed finite element methods, see \cite{DouglasWang1989}. However,
in the case of nonconforming finite elements, due to the reduced continuity of
trial and test functions, it becomes much more difficult to establish superconvergence
properties and related asymptotic error expansions. There are several superconvergence results on rectangular elements. In \cite{ChenLi1994,ShiJiang1997}, for the Wilson element \cite{BergerScott(1972)}, the superconvergence estimate of the gradient error on the centers of elements was obtained.
The essential point employed therein is that the Wilson element space can be
split into a conforming part and a nonconforming part. Thanks to the superconvergence estimate of the
consistency error, some superconvergence results of the nonconforming rotated $Q_1$ element \cite{RannacherTurek(1992)} and its variants were derived, see \cite{HuShi(2005),LinTobiska,MingShiXu2006}. As for the plate bending problem, there are only  few superconvergence results for nonconforming finite elements. In \cite{Chen2012}, Chen first established the supercloseness of the corrected interpolation of the incomplete biquadratic element \cite{Shi1986} on uniform rectangular meshes. By using similar corrected interpolations as in \cite{Chen2012}, Mao et al. \cite{MaoShi2009} first proved a half order superconvergence for the Morley element \cite{Morley(1968)} and the incomplete biquadratic nonconforming element on uniform rectangular meshes.  Based on the equivalence to the Stokes equations and a superconvergence result of Ye \cite{Ye2002} on the Crouzeix--Raivart element \cite{CrouzeixRaviart(1973)}, Huang et al. \cite{HuangHuang2013} derived the superconvergence for the Morley element, which was postprocessed by projecting the finite element solution to another finite element space on a coarser mesh \cite{Wang2000}.

In this paper, a new method is proposed to derive the superconvergence for nonconforming finite elements. The main idea is to explore some conformity of discrete stresses produced by nonconforming methods. Note that such conformity can not be obtained within original formulations for nonconforming elements. Fortunately, for the Crouzeix--Raviart element of the Poisson problem and the Morley element of the plate bending problem, it can be deduced by using the equivalences with the first order Raviart--Thomas element \cite{Marini1985} and the first order Hellan--Herrmann-Johnson element \cite{Arnold}, respectively. More precisely, based on these equivalences, we can translate the problem of superconvergence of nonconforming elements to the problem of superconvergence of mixed elements. Note that mixed elements are conforming methods within mixed formulations. This enables us to use superconvergence of mixed elements to establish superconvergence of nonconforming elements. In this way, it is able to overcome the main difficulty caused by nonconformity for the superconvergence analysis of nonconforming finite elements. In particular, a half order superconvergence by postprocessing is proved for both aforementioned two nonconforming elements on uniform triangulations. As a byproduct, the superconvergence is establised for the Hellan--Herrmann-Johnson element which is somehow missing in literature. Numerical tests are provided to demonstrate theoretical results.

The remaining paper is organized as follows. Section 2 proposes  the Poisson problem and the corresponding nonconforming and mixed finite elements. Section 3 presents the superconvergence result for the Raviart--Thomas element and proves the superconvergence result for the Crouzeix--Raviart element. Section 4 proposes the plate bending problem and the corresponding nonconforming and mixed finite elements. Section 5 proves the superconvergence result for the Hellan--Herrmann--Johnson element and the Morley element. Section 6 presents some numerical tests.
\section{The Poisson problem and its Crouzeix--Raviart element}
\label{sec:Poisson}
Throughout this paper, let $\Omega\subset\R^2$ be a polygonal domain. We recall some notations for Sobolev spaces (see \cite{Ciarlet1978}).
 For a subdomain $G$ of $\Omega$, let $P_m(G)$ be the space of polynomials of degree less than or equal to $m$ over $G$. $H^{s}(G)$ denotes the classical Sobolev space with norm $\|\cdot\|_{s,G}$ and the seminorm $|\cdot|_{s,G}$.  $W^{k,\infty}(G)$ denotes the classical Sobolev space with norm $\|\cdot\|_{k,\infty,G}$ and the seminorm $|\cdot|_{k,\infty,G}$.

 Given $f\in L^2(\Om)$,  the Poisson model  problem
finds $u\in H^1_0(\Om)$ such that
\begin{equation}\label{Poisson}
(\na u, \na v)=(f,v)\quad \text{ for all }v\in H^1_0(\Omega).
\end{equation}
By introducing an auxiliary variable $\sigma:=\nabla u$,  the problem can be formulated as the following equivalent
mixed problem which seeks $(\sigma, u)\in H(\div,\Omega)\times L^2(\Omega)$ such that

 \begin{equation}\label{MixedPoisson}
\begin{split}
 &(\sigma, \tau)+(u, \div \tau)=0\quad \text{ for any }\tau\in H(\div,\Omega),\\
 &(\div \sigma, v)=(-f, v)\quad \text{ for any }v\in L^2(\Omega).
\end{split}
 \end{equation}

 Suppose that $\bar{\Omega}$ is covered by uniform shape regular triangulations $\mathcal{T}$ consisting of triangles in two dimensions. $\mathcal{T}$ is said to be uniform if any two adjacent triangles of $\mathcal{T}$ form a parallelogram. $h$ denotes the diameter of the element $K\in\cT$. Let  $\cE$ denote the set of
edges of $\cT$, and  $\cE(\Om)$ denote the set of
all the interior edges, and  $\cE(\partial \Om)$ denote  the set of
all the  boundary edges. Given $e\in \cE$, let $\nu_e$ be the unit normal vector of $e$
and  $[\cdot]$ be jumps of piecewise functions over $e$, namely
$$
[v]:=v|_{K^+}-v|_{K^-}
$$
for piecewise functions $v$ and any two elements $K^+$ and $K^-$ which share the common edge $e$. Note that
$[\cdot]$  becomes  traces of functions  on $e$ for  boundary  edges
$e$.
Throughout the paper, an inequality $A\lesssim B$ replaces $A\leq CB$ with some multiplicative mesh--size independent constant $C>0$ .

 The Crouzeix--Raviart element \cite{CrouzeixRaviart(1973)} space  over $\mathcal{T}$ is defined by
 \begin{equation}
 W_{\rm CR}:=\begin{array}[t]{l}\big\{v\in L^2(\Om):
 v|_{K}\in P_1(K) \text{ for each }K\in \cT,
 \int_e[v]ds=0\text{ for all } e\in\cE(\Omega)
 \big\}\,,
 \end{array}\nn
\end{equation}
\begin{equation}
 V_{\rm CR}:=\begin{array}[t]{l}\big\{v\in W_{\rm CR}:\int_e vds=0 \text{ for all }e\in\cE(\partial\Omega)
 \big\}\,.
 \end{array}\nn
\end{equation}
The Crouzeix--Raviart element method of Problem \eqref{Poisson} finds  $u_{\rm CR}\in V_{\rm CR}$ such that
\begin{equation}\label{CRPoisson}
(\na_{\rm NC} u_{\rm CR}, \na_{\rm NC} v)=(f,v)\text{ for all }v\in V_{\rm CR}.
\end{equation}

To analyze the superconvergence of the Crouzeix--Raviart element, we  introduce the first order Raviart--Thomas element \cite{RaviartThomas1977} whose shape function
 space is
 $$
 {\rm RT}(K):=(P_0(K))^2+x P_0(K) \text{ for any }K\in\cT.
 $$
 Then the corresponding global finite element space reads
 \begin{equation}
 {\rm RT}(\cT):=\{\tau\in H(\div, \Om): \tau|_K\in {\rm RT}(K) \text{ for any }K\in\cT\}.
 \end{equation}
   To get a stable pair of space,  the piecewise constant space is proposed to
  approximate the displacement, namely,
  \begin{equation}
  {\rm U}_{\rm  RT}(\cT):=\{v\in L^2(\Om): v|_K\in P_0(K)\text{ for any }K\in\cT\}.
  \end{equation}
 The Raviart--Thomas element method of Problem \eqref{MixedPoisson} seeks $(\sigma_{\rm RT}, u_{\rm RT}) \in {\rm RT}(\cT)\times  {\rm U}_{\rm  RT}(\mathcal{T})$ such that
 \begin{equation}\label{DiscreteMixedPoisson}
\begin{split}
 &(\sigma_{\rm RT}, \tau)+(u_{\rm RT}, \div \tau)=0\quad \text{ for any }\tau\in {\rm RT}(\cT),\\
 &(\div \sigma_{\rm RT}, v)=(-f, v)\quad \text{ for any }v\in {\rm U}_{\rm  RT}(\cT).
\end{split}
 \end{equation}

Given $K\in\mathcal{T}$ and $f\in L^2(K)$, define $f_K=\frac{1}{|K|}\int_Kfdx$. Given $f\in L^2(\Omega)$, define the piecewise constant projection $\Pi_0f$ by
\begin{equation*}
  (\Pi_0f)|_K=f_K.
\end{equation*}
Because of the definition of $ {\rm U}_{\rm  RT}(\mathcal{T})$, $f$ in the second equation of \eqref{DiscreteMixedPoisson} can be replaced by $\Pi_0f$. We define the auxiliary method: Find  $\bar{u}_{\rm CR}\in V_{\rm CR}$ such that
\begin{equation}\label{MCRPoisson}
(\na_{\rm NC} \bar{u}_{\rm CR}, \na_{\rm NC} v)=(\Pi_0f,v)\text{ for all }v\in V_{\rm CR}.
\end{equation}

Note that this method differs from \eqref{CRPoisson} only by the presence of the projection in the right hand side. Marini \cite{Marini1985} proved its equivalence to the Raviart--Thomas element method \eqref{DiscreteMixedPoisson}:
\begin{equation}\label{equivalencePoisson}
\sigma_{\rm RT}|_K=\nabla \bar{u}_{\rm CR}|_K-\frac{f_K}{2}(x-{\rm Mid(K)})\quad x\in K\text{ for any $K\in\mathcal{T}$},
\end{equation}
where ${\rm Mid}(K)$ denotes the center of $K$.

Subtracting  \eqref{MCRPoisson} from \eqref{CRPoisson} with $v=u_{\rm CR}-\bar{u}_{\rm CR}$ yields that
\begin{equation*}
\begin{split}
(\nabla_{\rm NC}(u_{\rm CR}-\bar{u}_{\rm CR}),\nabla_{\rm NC}(u_{\rm CR}-\bar{u}_{\rm CR}))&=(f-\Pi_0f,u_{\rm CR}-\bar{u}_{\rm CR})\\
&=(f-\Pi_0f,u_{\rm CR}-\bar{u}_{\rm CR}-\Pi_0(u_{\rm CR}-\bar{u}_{\rm CR})).
\end{split}
\end{equation*}
Hence, the Poincar$\acute{\rm e}$ inequality from \cite{LaugesenSiudeja} yields
\begin{equation}\label{DifferencePoisson}
\|\nabla_{\rm NC}(u_{\rm CR}-\bar{u}_{\rm CR})\|_{0,\Omega}\leq \frac{h^2}{j^2_{1,1}}|f|_{1,\Omega},
\end{equation}
where $j_{1,1}= 3.8317$ denotes the first positive root of the Bessel function of the first kind.
\section{Superconvergence analysis of the Crouzeix--Raviart element}
\label{sec:poissonsuperconvergece}
In this section, we first present the superconvergence result of the Raviart--Thomas element by Brandts \cite{Brandts1994}. Then, based on this result and the equivalence \eqref{equivalencePoisson}, we derive the superconvergence result of the Crouzeix--Raviart element.
\subsection{The superconvergence result of the Raviart--Thomas element}
We introduce a result on Sobolev spaces in the following lemma, which describes the behavior of functions near the boundary. Define $\Omega_h$ as the subset of points in $\Omega$ having distance less that $h$ from the boundary:
\begin{equation*}
  \Omega_h=\{x\in\Omega:\exists y\in\partial\Omega,{\rm dist}(x,y)\leq h\}.
\end{equation*}
Then we have the following result, see \cite{Brandts1994,Lions1972}.
\begin{lemma}
\label{Boundarynorm}
For $v\in H^s(\Omega)$, where $0\leq s\leq\frac{1}{2}$, we have
\begin{equation*}
  \|v\|_{0,\Omega_h}\lesssim h^s\|v\|_{s,\Omega}.
\end{equation*}
\end{lemma}
Given $q\in (H^1(\Omega))^2$, define the interpolation operator $\Pi_{\rm RT}q\in{\rm RT}(\mathcal{T})$ by
\begin{equation*}
  \int_e(\Pi_{\rm RT}q-q)^T\nu_eds=0\quad\text{for all }e\in\mathcal{E}.
\end{equation*}
 Brandts gave the following superconvergence result of the Raviart--Thomas element, see \cite[Theorem 3.2]{Brandts1994} .
\begin{theorem}
\label{superRT}
Let $\sigma\in(H^2(\Omega))^2$ and $\sigma_{RT}$ be the solutions of \eqref{MixedPoisson} and \eqref{DiscreteMixedPoisson}, respectively.  There holds that
\begin{equation*}
 \|\sigma_{\rm RT}-\Pi_{\rm RT}\sigma\|_{0,\Omega}\lesssim h^{\frac{3}{2}}(\|\sigma\|_{\frac{3}{2},\Omega}+h^{\frac{1}{2}}|\sigma|_{2,\Omega}).
\end{equation*}
\end{theorem}
Furthermore, a post-processing mechanism was proposed in \cite{Brandts1994}, which when applied to the projection $\Pi_{\rm RT}q$ of a function $q\in (H^2(\Omega))^2$, will improve its approximation property. Given $q\in{\rm RT}(\cT)$, define function $K_hq\in (W_{\rm CR})^2$ as follows (see also \reffig{postprocessing}).
\begin{figure}[!ht]
  \centering
  \includegraphics[width=10cm]{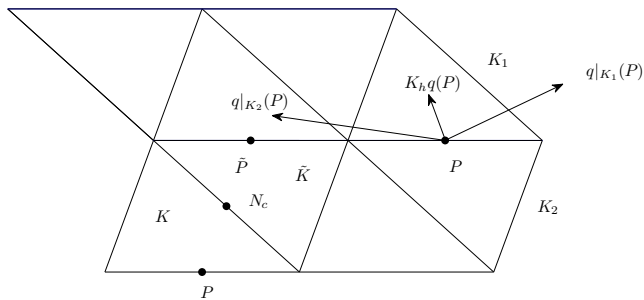}\\
  \caption{Post-processing a function $q\in{\rm RT}(\mathcal{T})$}\label{postprocessing}
\end{figure}
\begin{itemize}
  \item Given $e\in\cE(\Omega)$, suppose that $e=K_1\cap K_2$ and $P$ denotes the center of $e$. Let
\begin{equation*}
  K_hq(P)=\frac{1}{2}(q|_{K_1}(P)+q|_{K_2}(P)).
\end{equation*}
  \item Given $e\in\cE(\partial\Omega)$ and $e\subset\partial K$, there exists at least one $\tilde{K}\in\cT$ such that $N=K\cup\tilde{K}$ is a parallelogram. The straight line through the center $P$ of $e$ and the center $N_c$ of the parallelogram intersects the boundary of $N$ in another point $\tilde{P}$. Define
\begin{equation*}
  K_hq(P)=2K_hq(N_c)-K_hq(\tilde{P}).
\end{equation*}
\end{itemize}
Brandts \cite{Brandts1994} proved that the vector $K_h\Pi_{\rm RT}q$ is a higher order approximation of $q$ than $\Pi_{\rm RT}q$ itself.
\begin{theorem}
\label{postprocessinghighorder}
Suppose $q\in (H^2(\Omega))^2$, then there holds that
\begin{equation*}
  \|q-K_h\Pi_{\rm RT}q\|_{0,\Omega}\lesssim h^2|q|_{2,\Omega}.
\end{equation*}
\end{theorem}
Combining Theorem \ref{superRT}  and  Theorem \ref{postprocessinghighorder} concludes that the post-processing operator $K_h$ also improves the order of approximation of $\sigma_{\rm RT}$.
\begin{corollary}
\label{superconvergenceRT}
Let $\sigma\in(H^2(\Omega))^2$ and $\sigma_{RT}$ be the solutions of \eqref{MixedPoisson} and \eqref{DiscreteMixedPoisson}, respectively. There holds that
\begin{equation*}
  \|\sigma-K_h\sigma_{\rm RT}\|_{0,\Omega}\lesssim h^{\frac{3}{2}}(\|\sigma\|_{\frac{3}{2},\Omega}+h^{\frac{1}{2}}|\sigma|_{2,\Omega}).
\end{equation*}
\end{corollary}
\subsection{The superconvergence result of the Crouzeix--Raviart element}
 \begin{theorem}
  \label{The:CRsuperconvergence}
 Let $u\in H^3(\Omega)$ and $u_{\rm CR}$  be the solutions of \eqref{Poisson} and \eqref{CRPoisson}, respectively. Further, suppose that $f\in W^{1,\infty}(\Omega)$, then we have
  \begin{equation}
 \label{CRsuperconvergence}
 \|\nabla u-K_h\nabla_{\rm NC}u_{\rm CR}\|_{0,\Omega}\lesssim h^{\frac{3}{2}}(\|u\|_{\frac{5}{2},\Omega}+h^{\frac{1}{2}}|u|_{3,\Omega}+h^{\frac{1}{2}}|f|_{1,\infty,\Omega}).
 \end{equation}
 \end{theorem}
 \begin{proof}
 Using the equivalence equality \eqref{equivalencePoisson}, for $e=K_1\cap K_2$ and the center $P$ of $e$, there holds that
\begin{equation*}
|K_h(\nabla_{\rm NC}\bar{u}_{\rm CR}-\sigma_{\rm RT})(P)|=|\frac{f_{K_1}}{4}(P-{\rm Mid}(K_1))+\frac{f_{K_2}}{4}(P-{\rm Mid}(K_2))|
\end{equation*}
Since $K_1$ and $K_2$ form a parallelogram, we have $P-{\rm Mid}(K_1)={\rm Mid}(K_2)-P$. This yields that
\begin{equation*}
\begin{split}
|K_h(\nabla_{\rm NC}\bar{u}_{\rm CR}-\sigma_{\rm RT})(P)|&=\frac{1}{4}|(f_{K_1}-f_{K_2})(P-{\rm Mid}(K_1))|\\
&\lesssim h^2|f|_{1,\infty,\Omega}.
\end{split}
\end{equation*}
Suppose that $\phi_i,1\leq i\leq3$ denote the nodal basis functions on $K$ of $(W_{\rm CR})^2$. Hence, by the definition of $K_h$ and scaling arguments, there holds that
\begin{equation*}
 \|K_h(\nabla_{\rm NC}\bar{u}_{\rm CR}-\sigma_{\rm RT})\|^2_{0,K}\lesssim h^4|f|_{1,\infty,\Omega}^2\sum_{i=1}^3\|\phi_i\|^2_{0,K}\lesssim h^6|f|_{1,\infty,\Omega}^2.
 \end{equation*}
 Summing over all triangles $K\in\cT$ gives that
\begin{equation}
\label{superCR1}
\begin{split}
\|K_h(\nabla_{\rm NC}\bar{u}_{\rm CR}-\sigma_{\rm RT})\|_{0,\Omega}\lesssim h^2|f|_{1,\infty,\Omega}.
\end{split}
\end{equation}
Since $\nabla_{\rm NC}\bar{u}_{\rm CR}-\nabla_{\rm NC}u_{\rm CR}$ is a piecewise constant, the inverse estimate and \eqref{DifferencePoisson} yield that
\begin{equation}
\label{superCR2}
 \|K_h(\nabla_{\rm NC}\bar{u}_{\rm CR}-\nabla_{\rm NC}u_{\rm CR})\|_{0,\Omega}\lesssim\|\nabla_{\rm NC}(\bar{u}_{\rm CR}-u_{\rm CR})\|_{0,\Omega}\lesssim h^2|f|_{1,\Omega}\lesssim h^2|f|_{1,\infty,\Omega}.
\end{equation}
The triangle inequality plus Corollary \ref{superconvergenceRT}, \eqref{superCR1} and \eqref{superCR2} complete the proof.
\end{proof}

 \section{The plate bending problem and its Morley element}
Given $f\in L^2(\Omega)$, the plate bending model problem finds $u\in H^2_0(\Omega)$ such that
\begin{equation}\label{Plateequation}
  (\nabla^2u,\nabla^2v)=(f,v)\quad\text{for all }v\in H^2_0(\Omega).
\end{equation}
Given any space $V$, we define $(V)^4_{\rm s}$ as follows:
$$(V)^4_{\rm s}:=\{\tau=(\tau_{ij}),1\leq i\leq j\leq2:\tau_{ij}\in V,\tau_{12}=\tau_{21}\}.$$
Given $K\in\cT$, $\nu$ denotes the unit outward normal to $\partial K$ and $t$  the unit tangent to $\partial K$. Given $\tau\in (H^1(K))^4_{\rm s}$, we set
\begin{eqnarray*}
 M_{\nu\nu}(\tau) &=& \nu^T\tau\nu, \\
  M_{\nu t}(\tau)&=&\nu^T\tau t.
\end{eqnarray*}
By introducing an auxiliary variable $\sigma:=\nabla^2u$, the mixed formulation of \eqref{Plateequation} seeks $(\sigma,u)\in S\times D$, see \cite{Johnson},
 \begin{equation}\label{MixedPlate}
\begin{split}
 &(\sigma, \tau)+\sum_{K\in\mathcal{T}}-(\tau,\nabla^2u)_{L^2(K)}+\int_{\partial K}M_{\nu\nu}(\tau)\frac{\partial u}{\partial \nu}ds=0\quad \text{ for any }\tau\in S,\\
 &\sum_{K\in\mathcal{T}}-(\sigma,\nabla^2v)_{L^2(K)}+\int_{\partial K}M_{\nu\nu}(\sigma)\frac{\partial v}{\partial \nu}ds=(-f, v)\quad \text{ for any }v\in D,
\end{split}
 \end{equation}
 where
 \begin{equation*}
  \begin{split}
  S=&\{\tau\in(L^2(\Omega))^4_{\rm s}:\tau|_K\in (H^1(K))^4_{\rm s}\text{ for all }K\in\cT,\\
  &\text{ and }M_{\nu\nu}(\tau)\text{ is continuous across interelement edges}\},\\
  D=&\{v\in H^1_0(\Omega):v|_K\in H^2(K)\text{ for all }K\in\mathcal{T}\}.
  \end{split}
 \end{equation*}

 The Morley element space \cite{Morley(1968)} $V_{\rm M}$ over $\cT$ is defined by
 \begin{equation*}
 \begin{split}
   V_{\rm M}:=\big\{&v\in L^2(\Omega):v|_K\in P_2(K)\text{ for each }K\in\cT, v\text{ is continuous at each}\\&\text{interior vertex and vanishes on each boundary
   vertex},
   \int_e[\frac{\partial v}{\partial\nu_e}]ds=0\\
    & \text{for all }e\in\cE(\Omega),\text{ and } \int_e\frac{\partial v}{\partial\nu_e} ds=0\text{ for all }e\in\cE(\partial\Omega)\big\}.
   \end{split}
 \end{equation*}
The Morley element method of Problem \eqref{Plateequation} finds $u_{\rm M}\in V_{\rm M}$ such that
\begin{equation}\label{PlateDiscrete}
  (\nabla^2_{\rm NC}u_{\rm M},\nabla^2_{\rm NC}v)=(f,v)\quad\text{for all }v\in V_{\rm M}.
\end{equation}

To analyze the superconvergence of the Morley element, we introduce the first order Hellan--Herrmann--Johnson element \cite{Johnson}. Define
  \begin{equation*}
  \begin{split}
  {\rm HHJ}(\cT)=&\{\tau\in S:\tau|_K\in (P_0(K))^4_{\rm s}\text{ for any }K\in\mathcal{T}\},\\
  {\rm U}_{\rm HHJ}(\cT)=&\{v\in H^1_0(\Omega):v|_K\in P_1(K)\text{ for any }K\in\mathcal{T}\}.
  \end{split}
 \end{equation*}
The first order Hellan--Herrmann--Johnson element of Problem \eqref{MixedPlate} finds $(\sigma_{\rm HHJ},u_{\rm HHJ})\in {\rm HHJ}(\cT)\times  {\rm U}_{\rm HHJ}(\cT)$ such that
 \begin{equation}\label{MixedplateDiscrete}
\begin{split}
 &(\sigma_{\rm HHJ}, \tau)+\sum_{K\in\mathcal{T}}\int_{\partial K}M_{\nu\nu}(\tau)\frac{\partial u_{\rm HHJ}}{\partial\nu}ds=0\quad \text{ for any }\tau\in {\rm HHJ}(\cT),\\
 &\sum_{K\in\mathcal{T}}\int_{\partial K}M_{\nu\nu}(\sigma_{\rm HHJ})\frac{\partial v}{\partial\nu}ds=(-f, v)\quad \text{ for any }v\in  {\rm U}_{\rm HHJ}(\cT).
\end{split}
 \end{equation}

Given $v\in H^2_0(\Omega)\cup V_{\rm M}$, define the interpolation operator $\Pi_{\rm D}:H^2_0(\Omega)\cup V_{\rm M}\rightarrow {\rm U}_{\rm HHJ}(\cT)$ by
\begin{equation}\label{interP1con}
  \Pi_{\rm D}v(z)=v(z) \text{ for each vertex $z$ of } \cT.
\end{equation}
 Hence, we introduce the auxiliary method: The modified Morley element finds
 $\bar{u}_{\rm M}\in V_{\rm M}$ such that
\begin{equation}\label{PlateAxDiscrete}
  (\nabla^2_{\rm NC}\bar{u}_{\rm M},\nabla^2_{\rm NC}v)=(f,\Pi_{\rm D}v)\quad\text{for all }v\in V_{\rm M}.
\end{equation}

Arnold et al. \cite{Arnold} proved the following equivalence between the Hellan--Herrmann--Johnson element and the modified Morley element:
\begin{equation}
\label{equvilencePlate}
  \sigma_{\rm HHJ}=\nabla^2_{\rm NC}\bar{u}_{\rm M},u_{\rm HHJ}=\Pi_{\rm D}\bar{u}_{\rm M},
\end{equation}
and moreover
\begin{equation}
\label{DifferecePlate}
  \|\nabla_{\rm NC}^2(u_{\rm M}-\bar{u}_{\rm M})\|_{0,\Omega}\lesssim h^2\|f\|_{0,\Omega}.
\end{equation}
\section{Superconvergence analysis of the Morley element}
In this section, following the similar arguments for the Raviart--Thomas element in \cite{Brandts1994}, we prove the superconvergence  result of the Hellan--Herrmann--Johnson element. Then, based on this result and the equivalence \eqref{equvilencePlate}, we derive the superconvergence result  of the Morley element.
\subsection{The superconvergence result of the Hellan--Herrmann--Johnson element}
First we introduce the interpolation  operator $\Pi_{\rm HHJ}:S\rightarrow {\rm HHJ}(\cT)$ as in \cite{BrezziRaviart1978}:
 \begin{equation}\label{interDefHHJ}
  \int_eM_{\nu\nu}(\Pi_{\rm HHJ}\tau)ds=\int_eM_{\nu\nu}(\tau)ds\quad\text{for all } e\in\mathcal{E}.
 \end{equation}
 Moreover if $\tau\in (H^1(\Omega))^4_s$,
 \begin{equation}\label{interpolationHHJ}
  \|\tau-\Pi_{\rm HHJ}\tau\|_{0,\Omega}\lesssim h|\tau|_{1,\Omega}.
 \end{equation}
An integration by parts yields that the following Green's formulae holds for any $\tau\in (H^1(K))^4_{\rm s}$ and $v\in H^2(K)$,
\begin{equation}\label{greenformulae}
  \int_K\tau:\nabla^2vdx=-\int_K\div\tau\cdot\nabla vdx+\int_{\partial K}M_{\nu\nu}(\tau)\frac{\partial v}{\partial\nu}ds+\int_{\partial K}M_{\nu t}(\tau)\frac{\partial v}{\partial t}ds.
\end{equation}
We have the following result.
\begin{lemma}
Let $\sigma$ and $\sigma_{\rm HHJ}$ be the solutions of \eqref{MixedPlate} and \eqref{MixedplateDiscrete}, respectively. Then
\begin{equation}\label{HHJeq4}
 (\sigma_{\rm HHJ}-\sigma,\sigma_{\rm HHJ}-\Pi_{\rm HHJ}\sigma)=0.
\end{equation}
\end{lemma}
\begin{proof}
Let $\tau\in{\rm HHJ}(\cT),v\in {\rm U}_{\rm HHJ}(\cT)$  in \eqref{MixedPlate} and \eqref{MixedplateDiscrete}, which, together with \eqref{greenformulae}, yield that
 \begin{equation}\label{HHJeq1}
   \begin{split}
   (\sigma_{\rm HHJ}-\sigma,\tau)=&-\sum_{K\in\mathcal{T}}\int_{\partial K}M_{\nu\nu}(\tau)\frac{\partial(u_{\rm HHJ}-u)}{\partial\nu}ds-\sum_{K\in\mathcal{T}}(\tau,\nabla^2u)_{L^2(K)}\\
   =&\sum_{K\in\mathcal{T}}\int_{\partial K}M_{\nu t}(\tau)\frac{\partial(u_{\rm HHJ}-u)}{\partial t}ds,
      \end{split}
 \end{equation}
  and
 \begin{equation}\label{HHJeq2}
  \sum_{K\in\mathcal{T}}\int_{\partial K}M_{\nu\nu}(\sigma_{\rm HHJ}-\sigma)\frac{\partial v}{\partial\nu}ds=0.
 \end{equation}
  By  the definition of $\Pi_{\rm D}u$ in \eqref{interP1con}, since $M_{\nu t}(\tau)$ is constant on each edge of $K$, a combination of  \eqref{HHJeq1} and \eqref{greenformulae} leads to
 \begin{equation}\label{HHJeq3}
   \begin{split}
    (\sigma_{\rm HHJ}-\sigma,\tau)=&\sum_{K\in\mathcal{T}}\int_{\partial K}M_{\nu t}(\tau)\frac{\partial(u_{\rm HHJ}-\Pi_{\rm D}u)}{\partial t}ds\\
   =&-\sum_{K\in\mathcal{T}}\int_{\partial K}M_{\nu\nu}(\tau)\frac{\partial(u_{\rm HHJ}-\Pi_{\rm D}u)}{\partial\nu}ds.
   \end{split}
 \end{equation}
  Thanks to the definition of $\Pi_{\rm HHJ}$ in \eqref{interDefHHJ}, substituting  $\tau=\sigma_{\rm HHJ}-\Pi_{\rm HHJ}\sigma,v=u_{\rm HHJ}-\Pi_{\rm D}u$ into \eqref{HHJeq2} and \eqref{HHJeq3}, respectively, yields that
  \begin{equation*}
  \begin{split}
   (\sigma_{\rm HHJ}-\sigma,\sigma_{\rm HHJ}-\Pi_{\rm HHJ}\sigma)&=-\sum_{K\in\mathcal{T}}\int_{\partial K}M_{\nu\nu}(\sigma_{\rm HHJ}-\Pi_{\rm HHJ}\sigma)\frac{\partial(u_{\rm HHJ}-\Pi_{\rm D}u)}{\partial\nu}ds\\
   &=-\sum_{K\in\mathcal{T}}\int_{\partial K}M_{\nu\nu}(\sigma_{\rm HHJ}-\sigma)\frac{\partial(u_{\rm HHJ}-\Pi_{\rm D}u)}{\partial\nu}ds\\
   &=0.
      \end{split}
  \end{equation*}
  This completes the proof.
  \end{proof}
  \begin{lemma}
  \label{lemma:HHJ}
  Let $N$ be a parallelogram forming by  two triangles $K_1,K_2$. Then for all $r\in(P_1(N))^4_{\rm s}$, we have that
  \begin{equation*}
    \int_N(r-\Pi_{\rm HHJ}r)dx=0.
  \end{equation*}
  \end{lemma}
  \begin{proof}
  We may assume that $N$ is centered around the origin and, since $r=\Pi_{\rm HHJ}r$ whenever $r$ is constant, take $r\in(P_1(N))^4_{\rm s}$ zero at the origin
  and thus odd. But then $\Pi_{\rm HHJ}r$ is odd as well, which completes the proof.
  \end{proof}
 We recall some notations in \cite{Brandts1994}. Denote a parallelogram consisting of two triangles sharing a side with normal $f_i$ by $N_{f_i},(i=1,2,3)$. For each $i=1,2,3$, the domain $\Omega$ can be partitioned into parallelograms $N_{f_i}$ and some resulting boundary triangles which we denote by $T_{f_i}$. For an example of the definitions and notations concerning the triangulations, see \reffig{uniformtriangulation}.
  \begin{figure}[!ht]
  \centering
  \includegraphics[width=12cm,height=5cm]{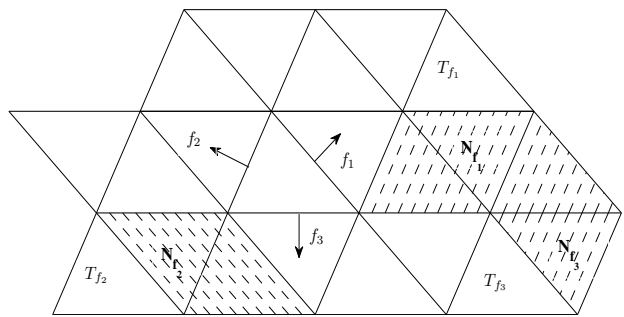}\\
  \caption{A uniform triangulation of $\Omega$}\label{uniformtriangulation}
\end{figure}
  \begin{theorem}
  \label{superHHJ}
Let $\sigma\in (H^2(\Omega))^4_{\rm s}$ and $\sigma_{\rm HHJ}$ be the solutions of \eqref{MixedPlate} and \eqref{MixedplateDiscrete}, respectively. Then
  \begin{equation*}
    \|\sigma_{\rm HHJ}-\Pi_{\rm HHJ}\sigma\|_{0,\Omega}\lesssim h^{\frac{3}{2}}(\|\sigma\|_{\frac{3}{2},\Omega}+h^{\frac{1}{2}}|\sigma|_{2,\Omega}).
  \end{equation*}
  \end{theorem}
  \begin{proof}
  First because of \eqref{HHJeq4}, we find that
  \begin{equation*}
   (\sigma_{\rm HHJ}-\Pi_{\rm HHJ}\sigma,\sigma_{\rm HHJ}-\Pi_{\rm HHJ}\sigma)=(\sigma_{\rm HHJ}-\Pi_{\rm HHJ}\sigma,\sigma-\Pi_{\rm HHJ}\sigma).
  \end{equation*}
  Let $\tau_{f_i}\in (P_0(K))^4_{\rm s},1\leq i\leq 3$ denote the basis functions, i.e., $M_{f_jf_j}(\tau_{f_i})=\delta_{ij}$. Then we have the following decomposition:
  \begin{equation*}
   \begin{split}
   (\sigma_{\rm HHJ}-\Pi_{\rm HHJ}\sigma,\sigma-\Pi_{\rm HHJ}\sigma)&=\sum_{K\in\cT}\int_K(\sigma_{\rm HHJ}-\Pi_{\rm HHJ}\sigma):(\sigma-\Pi_{\rm HHJ}\sigma)dx\\
      &=\sum_{K\in\cT}\int_K\sum_{i=1}^3M_{f_if_i}(\sigma_{\rm HHJ}-\Pi_{\rm HHJ}\sigma)\tau_{f_i}:(\sigma-\Pi_{\rm HHJ}\sigma)dx\\
      &=\sum_{i=1}^3I_i
    \end{split}
  \end{equation*}
  where
  \begin{equation*}
    I_i=\sum_{K\in\cT}\int_KM_{f_if_i}(\sigma_{\rm HHJ}-\Pi_{\rm HHJ}\sigma)\tau_{f_i}:(\sigma-\Pi_{\rm HHJ}\sigma)dx.
  \end{equation*}
  Since $M_{f_if_i}(\sigma_{\rm HHJ}-\Pi_{\rm HHJ}\sigma)$ is continuous and constant on $N_{f_i}$, and since $\tau_{f_i}$ is constant on $N_{f_i}$, rewriting the sum $I_i$ as a sum over parallelogram $N_{f_i}$, boundary triangles $T_{f_i}$, we find:
    \begin{equation}\label{itemtotal}
    \begin{split}
   |I_i|\leq&\sum_{N_{f_i}}\big|M_{f_if_i}(\sigma_{\rm HHJ}-\Pi_{\rm HHJ}\sigma)\tau_{f_i}:\int_{N_{f_i}}(\sigma-\Pi_{\rm HHJ}\sigma)dx\big|\\
   &+\sum_{T_{f_i}}\big|\int_{T_{f_i}}M_{f_if_i}(\sigma_{\rm HHJ}-\Pi_{\rm HHJ}\sigma)\tau_{f_i}:(\sigma-\Pi_{\rm HHJ}\sigma)dx\big|.
  \end{split}
  \end{equation}
 Denote $\partial\Omega_{f_i}$ the union of the boundary triangle $T_{f_i}$. In bounding \eqref{itemtotal} we use the Cauchy-Schwarz inequality and the estimate
  \begin{equation*}
    |M_{f_if_i}(\sigma_{\rm HHJ}-\Pi_{\rm HHJ}\sigma)\tau_{f_i}|\lesssim h^{-1}\|\sigma_{\rm HHJ}-\Pi_{\rm HHJ}\sigma\|_{0,N_{f_i}}.
  \end{equation*}
  which results in
  \begin{equation}\label{itemtotal1}
  \begin{split}
   |I_i|\lesssim& h^{-1}\|\sigma_{\rm HHJ}-\Pi_{\rm HHJ}\sigma\|\left(\sum_{N_{f_i}}\bigg|\int_{N_{f_i}}(\sigma-\Pi_{\rm HHJ}\sigma)dx\bigg|^2\right)^{\frac{1}{2}}\\
   &+\|\sigma_{\rm HHJ}-\Pi_{\rm HHJ}\sigma\|_{0,\partial\Omega_{f_i}}\|\sigma-\Pi_{\rm HHJ}\sigma\|_{0,\partial\Omega_{f_i}}.
   \end{split}
  \end{equation}
  Define the linear functional $\mathcal{F}$ on $(H^2(N_{f_i}))^4_{\rm s}$ by
  \begin{equation*}
  \mathcal{F}(\tau)=\int_{N_{f_i}}(\tau-\Pi_{\rm HHJ}\tau)dx,\tau\in(H^2(N_{f_i}))^4_{\rm s}.
  \end{equation*}
  For this functional, the Cauchy-Schwarz inequality and \eqref{interpolationHHJ} yield:
  \begin{equation*}
|\mathcal{F}(\tau)|\lesssim h\|\tau-\Pi_{\rm HHJ}\tau\|_{0,N_{f_i}}\lesssim h^2|\tau|_{1,N_{f_i}}.
  \end{equation*}
  Since each parallelogram $N_{f_i}$ is a translate of the parallelogram $N$ of Lemma \ref{lemma:HHJ}, one can find that $(P_1(N))^4_{\rm s}\subset{\rm Ker}(\mathcal{F})$, and a standard application of the Bramble-Hilbert lemma \cite{BrambleHilbert1970} gives
    \begin{equation}\label{functionproperty}
|\mathcal{F}(\tau)|\lesssim h^3|\tau|_{2,N_{f_i}}\quad\text{for all }\tau\in(H^2(N_{f_i}))^4_{\rm s}.
  \end{equation}
 Combing \eqref{itemtotal1}, \eqref{interpolationHHJ} and \eqref{functionproperty}, we conclude that
      \begin{equation*}
    \begin{split}
   |I_i|\lesssim\|\sigma_{\rm HHJ}-\Pi_{\rm HHJ}\sigma\|_{0,\Omega}(h^2|\sigma|_{2,\Omega}+h|\sigma|_{1,\partial\Omega_{f_i}}).
   \end{split}
  \end{equation*}
Lemma \ref{Boundarynorm} implies that
  \begin{equation*}
 |\sigma|_{1,\partial\Omega_{f_i}}\leq|\sigma|_{1,\Omega_{h}}\lesssim h^{\frac{1}{2}}\|\sigma\|_{\frac{3}{2},\Omega}.
  \end{equation*}
  This completes the estimate of $|I_i|$.
  \end{proof}

  We use a similar post-processing mechanism  as in Section \ref{sec:poissonsuperconvergece} and still denote the post-processing operator as $K_h$. Thus given $\tau\in{\rm HHJ}(\cT)$, $K_h\tau\in(W_{\rm CR})^4_{\rm s}$ is similar defined as in Section \ref{sec:poissonsuperconvergece}. Following  the idea of \cite[Theorem 5.1]{Brandts1994}, we can prove the following result.
  \begin{theorem}
  \label{postHHJ}
Let $\tau\in(H^2(\Omega))^4_{\rm s}$. Then for $K_h\Pi_{\rm HHJ}\tau\in (W_{\rm CR})^4_{\rm s}$, we have
  \begin{equation*}
    \|\tau-K_h\Pi_{\rm HHJ}\tau\|_{0,\Omega}\lesssim h^2|\tau|_{2,\Omega}.
  \end{equation*}
  \end{theorem}
  \begin{proof}
  First, let $r\in(P_1(\tilde{K}))^4_{\rm s}$, where $\tilde{K}$ is the union of $K$ and the triangles sharing a edge with $K$. Then, using the same arguments as in Lemma \ref{lemma:HHJ} we find that
  \begin{equation}\label{postHHJ:1}
    K_h\Pi_{\rm HHJ}r=r\text{ on }K\quad\text{for all }r\in(P_1(\tilde{K}))^4_{\rm s}.
  \end{equation}
For all $\tau\in(H^2(\Omega))^4_{\rm s}$, since $K_h\Pi_{\rm HHJ}\tau$ is a linear function on $K$, there holds that
  \begin{equation}\label{postHHJ:2}
    \|K_h\Pi_{\rm HHJ}\tau\|_{0,\infty,K}\lesssim\|\Pi_{\rm HHJ}\tau\|_{0,\infty,\tilde{K}}.
  \end{equation}
  Since the interpolation $\Pi_{\rm HHJ}\tau$ is constant on $K$,
  and since the angles between the normals of the edges of $K$ are
  bounded away from $0$ and $\pi(-\pi)$, we have
  \begin{equation}\label{postHHJ:3}
    \|\Pi_{\rm HHJ}\tau\|_{0,\infty,K}\lesssim\sum^3_{j=1}|M_{f_jf_j}(\Pi_{\rm HHJ}\tau)|\lesssim\sum^3_{j=1}\|M_{f_jf_j}(\tau)\|_{0,\infty,\partial K_j}\lesssim\|\tau\|_{0,\infty,K}.
  \end{equation}
  From \eqref{postHHJ:2} and \eqref{postHHJ:3}, we conclude that
  \begin{equation*}
\|K_h\Pi_{\rm HHJ}\tau\|_{0,\infty,K}\lesssim\|\tau\|_{0,\infty,\tilde{K}},
  \end{equation*}
  so that using \eqref{postHHJ:1}, for all $r\in(P_1(\tilde{K}))^4_{\rm s}$
  \begin{equation*}
  \begin{split}
   \|\tau-K_h\Pi_{\rm HHJ}\tau\|_{0,K}&\lesssim h\|\tau-K_h\Pi_{\rm HHJ}\tau\|_{0,\infty,K}\lesssim h\|({\rm I}-K_h\Pi_{\rm HHJ})(\tau-r)\|_{0,\infty,K}\\
  & \lesssim h\|\tau-r\|_{0,\infty,\tilde{K}}.
     \end{split}
  \end{equation*}
 The interpolation theory in Sobolev spaces (see \cite{Ciarlet1978}) shows that
  \begin{equation*}
    \inf\{r\in(P_1(\tilde{K}))^4_{\rm s}:\|\tau-r\|_{0,\infty,\tilde{K}}\}\lesssim h|\tau|_{2,\tilde{K}},
  \end{equation*}
 which yields
  \begin{equation}\label{postHHJ:4}
  \|\tau-K_h\Pi_{\rm HHJ}\tau\|_{0,K}\lesssim h^2|\tau|_{2,\tilde{K}}.
  \end{equation}
 Hence, squaring \eqref{postHHJ:4} and summing over all triangles $K\in\cT$ complete the proof.
  \end{proof}
  A combination of  the superconvergence result and Theorem \ref{postHHJ}, concludes that the post-processing operator $K_h$ also improves the order
 of approximation of $\sigma_{\rm HHJ}$.
\begin{corollary}
Let $\sigma\in (H^2(\Omega))^4_{\rm s}$ and $\sigma_{\rm HHJ}$ be the solutions of \eqref{MixedPlate} and \eqref{MixedplateDiscrete}, respectively. There holds that
  \begin{equation}\label{superconvergenceHHJ}
    \|\sigma-K_h\sigma_{\rm HHJ}\|_{0,\Omega}\lesssim h^{\frac{3}{2}}(\|\sigma\|_{\frac{3}{2},\Omega}+h^{\frac{1}{2}}|\sigma|_{2,\Omega}).
  \end{equation}
  \end{corollary}
  \subsection{The superconvergence result of the Morley element}
  \begin{theorem}\label{The:Platesuperconvergence}
  Let $u\in H^4(\Omega)$ and $u_{\rm M}$ be the solutions \eqref{Plateequation} and \eqref{PlateDiscrete}, respectively. Then we have
    \begin{equation}\label{Platesuperconvergence}
    \|\nabla^2u-K_h\nabla^2_{\rm NC}u_{\rm M}\|_{0,\Omega}\lesssim h^{\frac{3}{2}}(\|u\|_{\frac{7}{2},\Omega}+h^{\frac{1}{2}}|u|_{4,\Omega}+h^{\frac{1}{2}}\|f\|_{0,\Omega}).
  \end{equation}
  \end{theorem}
  \begin{proof}
  The triangle inequality plus the equivalence \eqref{equvilencePlate} and the inverse estimate give that
 \begin{equation*}
 \begin{split}
   \|\nabla^2u-K_h\nabla^2_{\rm NC}u_{\rm M}\|_{0,\Omega}&\lesssim\|\nabla^2u-K_h\nabla^2_{\rm NC}\bar{u}_{\rm M}\|_{0,\Omega}+\|K_h(\nabla^2_{\rm NC}u_{\rm M}-\nabla^2_{\rm NC}\bar{u}_{\rm M})\|_{0,\Omega}\\
   &\lesssim \|\sigma-K_h\sigma_{\rm HHJ}\|_{0,\Omega}+\|\nabla^2_{\rm NC}u_{\rm M}-\nabla^2_{\rm NC}\bar{u}_{\rm M}\|_{0,\Omega}.
   \end{split}
 \end{equation*}
Thus \eqref{superconvergenceHHJ} and \eqref{DifferecePlate} complete the proof.
  \end{proof}
We can only prove a half order superconvergence in Theorem \ref{The:Platesuperconvergence}. Under the same assumptions as in \cite[Theorem 4.4]{MaoShi2009}, we give the following one order superconvergence result.
  \begin{theorem}
  Under the assumption of Theorem \ref{The:Platesuperconvergence}, and further suppose that $\nabla^3u|_{\partial\Omega}=0$,  then we have
 \begin{equation*}
   \|\nabla^2u-K_h\nabla^2_{\rm NC}u_{\rm M}\|_{0,\Omega}\lesssim h^2(|u|_{4,\Omega}+\|f\|_{0,\Omega}).
 \end{equation*}
  \end{theorem}
   \begin{proof}
 We reconsider the estimate of the second term on the right hand of \eqref{itemtotal} in Theorem \ref{superHHJ}. Since $\nabla^3u|_{\partial\Omega}=0$, i.e., $\nabla\sigma|_{\partial\Omega}=0$, the Poincar$\rm\acute{e}$ inequality and scaling arguments show that
\begin{equation*}
  \big|\int_{T_{f_i}}(\sigma-\Pi_{\rm HHJ}\sigma)dx\big|\lesssim h^2|\sigma|_{1,T_{f_i}}\lesssim h^3|\sigma|_{2,T_{f_i}}.
\end{equation*}
 Hence, this results in one order superconvergence as follows:
  \begin{equation*}
    \|\sigma_{\rm HHJ}-\Pi_{\rm HHJ}\sigma\|_{0,\Omega}\lesssim h^2|\sigma|_{2,\Omega}.
  \end{equation*}
  Thus this completes the proof.
  \end{proof}
  \section{Numerical Tests}
  In this section, we present some numerical tests to confirm some of the theoretical analyses in the previous sections.
  \subsection{The Poisson problem}
Suppose domain $\Omega$ is a square, see \reffig{squaredomain}. Consider the following Poisson problem
\begin{equation*}
  -\Delta u=f\quad\text{in }\Omega
\end{equation*}
with $u\in H^1_0(\Omega)$. The exact solution is
\begin{equation*}
  u(x_1,x_2)=\sin\pi x_1\sin\pi x_2.
\end{equation*}
\begin{figure}[!ht]
  \centering
  \includegraphics[width=6cm]{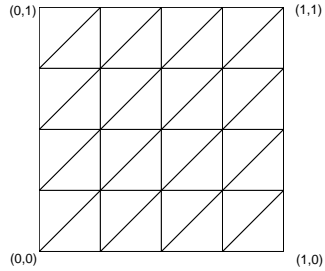}\\
  \caption{Square domain with uniform triangulations}\label{squaredomain}
\end{figure}

We compare the error $\|\nabla u-\nabla_{\rm NC}u_{\rm CR}\|_{0,\Omega}$ and the post-processing error $\|\nabla u-K_h\nabla_{\rm NC}u_{\rm CR}\|_{0,\Omega}$. The corresponding computational results are showed in \reffig{figpoisson} and listed in Table \ref{poissonsup}. It can be seen that the $O(h^{\frac{3}{2}})$ convergence rate $\|\nabla u-K_h\nabla_{\rm NC}u_{\rm CR}\|_{0,\Omega}$ in Theorem \ref{The:CRsuperconvergence} is verified by the numerical results. However, the numerical results indicate that the convergence rate is $O(h^2)$. So that the order proved in Theorem \ref{The:CRsuperconvergence} may be suboptimal.
 \begin{figure}[!ht]
  \centering
  \includegraphics[width=10cm,height=5.5cm]{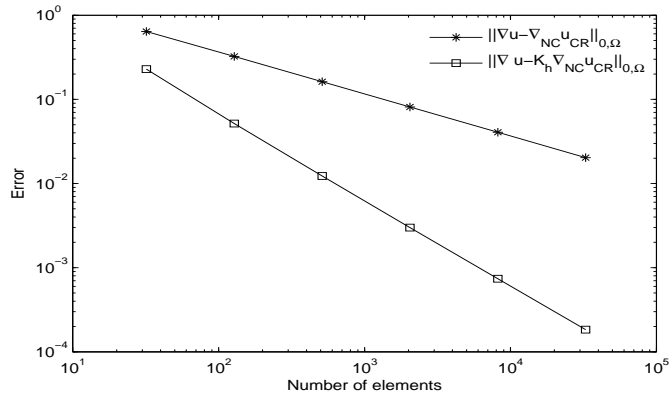}\\
  \caption{Convergence of the Crouzeix-Raviart element}\label{figpoisson}
\end{figure}
\begin{table}[!ht]
  \caption{Convergence of the Crouzeix-Raviart element}\label{poissonsup}
  \centering
\begin{tabular}{|c|c|c|c|c|}
\hline
Number of elements  &$\|\nabla u-\nabla_{\rm NC}u_{\rm CR}\|_{0,\Omega}$  &    $Rate$       &    $\|\nabla u-K_h\nabla_{\rm NC}u_{\rm CR}\|_{0,\Omega}$ &  $Rate$              \\
\hline
       $8\times4$    & 6.4104E-01 &            & 2.2880E-01 &            \\
\hline
   $16\times 8$        & 3.2395E-01 &    0.9847  & 5.1669E-02 &    2.1467  \\
\hline
     $32\times 16$       & 1.6241E-01 &    0.9961  & 1.2286E-02 &    2.0723  \\
\hline
   $64\times 32$        & 8.1259E-02 &    0.9990  & 2.9936E-03 &    2.0370  \\
\hline
    $128\times64$       & 4.0636E-02 &    0.9998  & 7.3852E-04 &    2.0192  \\
\hline
  $256\times128$         & 2.0319E-02 &    0.9999  & 1.8337E-04 &    2.0098  \\
\hline
\end{tabular}
\end{table}

 \subsection{The plate bending problem}
Suppose domain $\Omega$ is a parallelogram, see \reffig{paralledomain}. Consider the following plate bending problem
\begin{equation*}
  \Delta^2 u=f\quad\text{in }\Omega
\end{equation*}
with $u\in H^2_0(\Omega)$. The exact solution is
\begin{equation*}
  u(x_1,x_2)=(x_1-\sqrt{3}x_2)^2(x_1-\sqrt{3}x_2-2)^2x_2^2(\frac{\sqrt{3}}{2}-x_2)^2.
\end{equation*}
\begin{figure}[!ht]
  \centering
  \includegraphics[width=8cm]{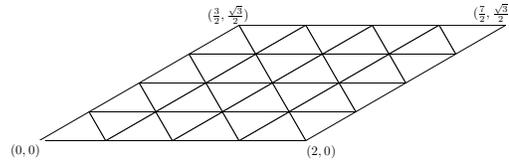}\\
  \caption{Parallelogram domain with uniform triangulations}\label{paralledomain}
\end{figure}

We compare the error $\|\nabla^2 u-\nabla_{\rm NC}^2u_{\rm M}\|_{0,\Omega}$ and the post-processing error $\|\nabla^2 u-K_h\nabla^2_{\rm NC}u_{\rm M}\|_{0,\Omega}$. The corresponding computational results are showed in \reffig{figplate} and listed in Table \ref{platesuper}. It can be seen that the $O(h^{\frac{3}{2}})$ convergence rate $\|\nabla^2 u-K_h\nabla^2_{\rm NC}u_{\rm M}\|_{0,\Omega}$ in Theorem \ref{The:Platesuperconvergence} is verified by the numerical results. However, the numerical results still indicate that the convergence rate is $O(h^2)$. So that the order proved in Theorem \ref{The:Platesuperconvergence} may be suboptimal.
 \begin{figure}[!ht]
  \centering
  \includegraphics[width=10cm,height=5.5cm]{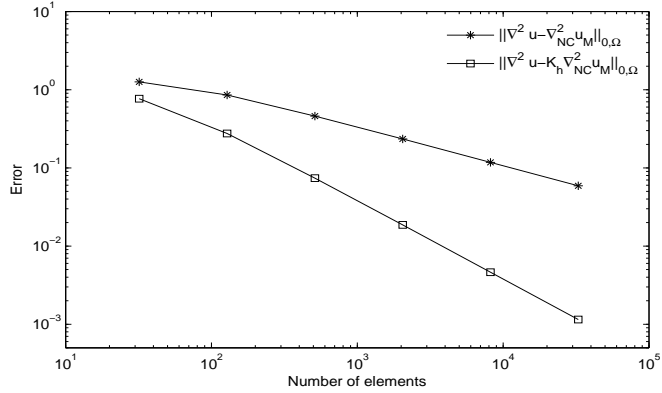}\\
  \caption{Convergence of the Morley element}\label{figplate}
\end{figure}
\begin{table}[!ht]
  \caption{Convergence of the Morley element}\label{platesuper}
  \centering
\begin{tabular}{|c|c|c|c|c|}
\hline
Number of elements  &$\|\nabla^2u-\nabla_{\rm NC}^2u_{\rm M}\|_{0,\Omega}$   &   $Rate$&    $\|\nabla^2u-K_h\nabla_{\rm NC}^2u_{\rm M}\|_{0,\Omega}$          &    $Rate$  \\
\hline
         $8\times4$        & 1.2599E+00 &            &    7.6681E-01  &            \\
\hline
   $16\times 8$         & 8.5516E-01 &    0.5591  &    2.7553E-01  &    1.4766  \\
\hline
    $32\times 16$           & 4.6008E-01 &    0.8943  &    7.3946E-02  &    1.8977  \\
\hline
    $64\times 32$          & 2.3428E-01 &    0.9736  &    1.8627E-02  &    1.9891  \\
\hline
         $128\times64$         & 1.1768E-01 &    0.9934  &   4.6311E-03  &    2.0080  \\
\hline
   $256\times128$           & 5.8909E-02 &    0.9983  &    1.1506E-03  &    2.0090  \\
\hline
\end{tabular}
\end{table}

\end{document}